\documentclass[11pt]{amsart}

\usepackage{amsmath,amssymb,amsfonts}
\usepackage{hhline}
\usepackage{longtable}

\newtheorem{Le}{Lemma}[section]

\newtheorem{Th}[Le]{Theorem}

\theoremstyle{definition}

\theoremstyle{remark}

\newtheorem{Rem}[Le]{Remark}

\allowdisplaybreaks
\begin{document}

\title[Non-symplectic automorphisms and Calabi-Yau orbifolds]{K3 surfaces with non-symplectic 
automorphisms of order three and Calabi-Yau orbifolds}
\author{Frank Reidegeld}

\maketitle

\begin{abstract}
Let $S$ be a K3 surface that admits a non-symplectic automorphism
$\rho$ of order $3$. We divide $S\times \mathbb{P}^1$ by $\rho\times\psi$
where $\psi$ is an automorphism of order $3$ of $\mathbb{P}^1$. There exists
a threefold ramified cover of a partial crepant resolution of the quotient that
is a Calabi-Yau or\-bi\-fold. We compute the Euler characteristic of our 
examples and obtain values ranging from $30$ to $219$. 
\end{abstract}

\section{Introduction}

The motivation for this article are various constructions of Calabi-Yau
ma\-nifolds that involve non-symplectic automorphisms of K3 surfaces.
An automorphism of a K3 surface $S$ is called non-symplectic if it
acts non-trivially on $H^{2,0}(S)$. The oldest example of such a
construction was proposed by Borcea \cite{Bor} and Voisin \cite{Voi}. 
Their idea is to divide the product of $S$ and an elliptic curve $E$ 
by the product $\rho\times\psi$ of two involutions, where $\rho$
shall be non-symplectic. The quotient $(S\times E)/(\rho\times\psi)$ 
is an orbifold whose singularities are the product of ordinary double 
points and the curves in the fixed locus of $\rho$. These singularities 
can be blown up and the authors obtain smooth Calabi-Yau threefolds.

It is a natural question if it is possible to construct Calabi-Yau manifolds
with help of non-symplectic automorphisms of higher order. These 
automorphisms are classified in \cite{ArSa,Taki} for order $3$
and for arbitrary prime order in \cite{ArSaTa}. Rohde \cite{Rohde1,Rohde2}
and Dillies \cite{Dillies} divide a product of a K3 surface and an elliptic
curve by a cyclic group that is generated by $\rho\times\psi$ where $\rho$ 
and $\psi$ are of order $3$. The crepant 
resolutions of those quotients are smooth Calabi-Yau manifolds. 
Garbagnati \cite{Garbagnati} obtains Calabi-Yau threefolds by
dividing a suitable product of a K3 surface and an elliptic curve by an
automorphism of order $4$. Further examples where
$\rho\times\psi$ is of order $4$ or $6$ can be found in \cite{CG}.
Since an automorphism of an elliptic curve has order $2$, $3$, $4$
or $6$, there are no further possible values of the order of $\rho\times
\psi$.  

Kovalev and Lee \cite{KoLe} replace $E$ by the Riemann
sphere $\mathbb{P}^1$ and $\psi$ by the map $\psi(z):=-z$. The
quotient $(S\times\mathbb{P}^1)/(\rho\times\psi)$, where $\rho$
is a non-symplectic involution, is an orbifold whose singularities
can be blown up as in the Borcea-Voisin construction. The 
resulting threefold $X$ is not a Calabi-Yau manifold. Nevertheless, one can
remove an anti-canonical K3 divisor $D$ and there exists an asymptotically
cylindrical Ricci-flat K\"ahler metric on $X\setminus D$. In other words,
$X\setminus D$ is a non-compact Calabi-Yau manifold. The manifolds that 
are constructed by this method can be used as building blocks for compact 
$G_2$-manifolds. 

If we divide $S\times \mathbb{P}^1$ by $\rho\times\psi$ where
$\rho$ is of order $3$ and $\psi(z) := \exp{(-\tfrac{2\pi i}{3})} z$, the 
K3 divisor is not anti-canonical anymore and we cannot apply the 
method of \cite{KoLe}. Moreover, the singularities of the quotient are 
not Gorenstein. We are nevertheless able to construct 
a threefold ramified cover that is a compact Calabi-Yau orbifold
and thus carries a Ricci-flat K\"ahler orbifold metric. Our method
is motivated by an article of Cynk \cite{Cynk}
where Calabi-Yau manifolds that are ramified covers of 
Fano varieties are constructed. 

This article is organized as follows. In the next section, we introduce the
most important facts on non-symplectic automorphisms of order $3$ and
their classification. After that, we describe our construction in detail
and prove that it yields indeed a Calabi-Yau orbifold. In the fourth section, 
we compute the Euler characteristic of our examples.

\section{Non-symplectic automorphisms of order $3$}
\label{NonSymp}

A smooth compact complex surface $S$ with trivial canonical bundle 
and $\pi_1(S)=0$ is called a \emph{K3 surface}. The cohomology
group $H^2(S,\mathbb{Z})$ together with the intersection form is a
lattice that is isomorphic to $L:=H^3 \oplus (-E_8)^2$, where
$H$ is the hyperbolic plane lattice and $-E_8$ is the root lattice
with negative signature that corresponds to the root system $E_8$.
We call $L$ the \emph{K3 lattice}. A K3 surface together with a
lattice isometry $\phi:H^2(S,\mathbb{Z}) \rightarrow L$ is
called a \emph{marked K3 surface}. 

An automorphism $\rho:S\rightarrow S$ is called \emph{of order $3$} if 
$\rho\neq\text{Id}_S$ but $\rho^3=\text{Id}_S$. Moreover, it is called 
\emph{non-symplectic} if the pull-back $\rho^{\ast}$ acts on the one-dimensional 
complex space $H^{2,0}(S)$ as multiplication by $\zeta_3:=\exp{(\tfrac{2\pi i}{3})}$. 
The non-symplectic automorphisms of order $3$ are classified in \cite{ArSa,Taki}. 
We sum up these classification results and refer the reader to \cite{ArSa,Taki} for 
further information. 

The map $\phi\circ\rho^{\ast} \circ \phi^{-1}$ is an isometry of $L$ 
that we denote also by $\rho^{\ast}$. Let 

\begin{equation}
L^{\rho}:=\{x\in L | \rho^{\ast}(x)=x\}
\end{equation}

be the \emph{fixed lattice}. The \emph{discriminant group} of $L^{\rho}$ is
defined as $L^{\rho\ast}/L^{\rho}$ where $L^{\rho\ast} = 
\text{Hom}(L^{\rho},\mathbb{Z})$. It is a finite group that is isomorphic
to $\mathbb{Z}_3^a$ for an $a\in\mathbb{N}_0$. The rank of $L^{\rho}$ 
is even and we define $m:=\tfrac{1}{2} L^{\rho\perp}$. In \cite{ArSa,Taki} 
the following results are proven:

\begin{Th} Let $L^{\rho}$ be the fixed lattice of a K3 surface with
a non-symplectic automorphism of order $3$. The isomorphism
type of $L^{\rho}$ is determined by the invariants $(m,a)$ that are defined
above. 
\end{Th}

\begin{Th} 
\label{ClassThm}
Let $(m,a)$ be a pair of natural numbers from the following list:

\begin{center}
\begin{tabular}{lll}
$(1,1)$ & \qquad $(5,3)$ & \qquad $(7,5)$ \\
$(2,0)$ & \qquad $(5,5)$ & \qquad $(7,7)$ \\
$(2,2)$ & \qquad $(6,0)$ & \qquad $(8,2)$ \\
$(3,1)$ & \qquad $(6,2)$ & \qquad $(8,4)$ \\
$(3,3)$ & \qquad $(6,4)$ & \qquad $(9,1)$ \\
$(4,2)$ & \qquad $(6,6)$ & \qquad $(9,3)$ \\
$(4,4)$ & \qquad $(7,1)$ & \qquad $(10,0)$ \\
$(5,1)$ & \qquad $(7,3)$ & \qquad $(10,2)$ \\
\end{tabular}
\end{center}

Then there exists a K3 surface with a non-symplectic automorphism 
of order $3$ such that its fixed lattice has invariants $(m,a)$. Conversely, 
the invariants $(m,a)$ of any K3 surface with a non-symplectic automorphism
of order $3$ are contained in the above list.
\end{Th}

The authors of \cite{ArSa,Taki} show that the invariants $(m,a)$ 
determine the topological type of the fixed locus $S^{\rho}:=
\{p\in S | \rho(p)=p \}$.

\begin{Th} Let $S$ be a K3 surface with a non-symplectic automorphism
$\rho$ of order $3$. The fixed locus of $\rho$ is non-empty and the disjoint
union of  

\begin{itemize}
    \item three isolated points if $m=a=7$,

    \item $n$ isolated points, $k$ smooth rational curves, and one smooth
    curve of genus $g$ where  

    \begin{itemize}
        \item[$\ast$] $n=10-m$,
        \item[$\ast$] $k=6-\tfrac{m+a}{2}$,
        \item[$\ast$] $g = \tfrac{m-a}{2}$.
    \end{itemize}
\end{itemize}

If $p\in S$ is an isolated fixed point, there exist local holomorphic 
coordinates around $p$ such that the action of $\rho$ on the 
tangent space $T_p S$ is given by 

\[
\left(\,
\begin{array}{cc}
\zeta_3^2 & 0 \\
0 & \zeta_3^2 \\
\end{array}
\,\right)
\]

If $p$ lies on a curve in $S^{\rho}$, the action of $\rho$ on 
$T_p S$ is given by

\[
\left(\,
\begin{array}{cc}
\zeta_3 & 0 \\
0 & 1 \\
\end{array}
\,\right)
\]
\end{Th}

\section{The construction of the Calabi-Yau manifolds}

The aim of this section is to prove the following theorem. 

\begin{Th}
\label{MainThm} Let $S$ be a K3 surface that admits a non-symplectic
automorphism $\rho$ of order $3$. Moreover, let $\psi:\mathbb{P}^1
\rightarrow \mathbb{P}^1$ be defined by $\psi(z):= \zeta^2_3 z$, where
$z$ is the usual complex coordinate on the Riemann sphere. We 
construct a complex orbifold $X$ by carrying out the following three 
steps. 

\begin{enumerate}
    \item We divide $S\times \mathbb{P}^1$ by the cyclic group that is 
    generated by $\rho\times\psi$. The quotient shall be denoted 
    by $Z$ and the projection map by $p:S\times\mathbb{P}^1 
    \rightarrow Z$.

    \item Let $\phi:X_0\rightarrow Z$ be a crepant resolution of
    the singularities on $p(S\times \{0\})$. Moreover, let 
    $D_0:= \phi^{-1}(p(S\times\{z\}))$, where $z\in\mathbb{P}^1
    \setminus\{0,\infty\}$.

    \item Finally, let $\pi:X\rightarrow X_0$ be a suitable threefold
    ramified cover with branching locus $D_0$. 
\end{enumerate}

$X$ is a Calabi-Yau orbifold, i.e. $c_1(X)\in H^2(X,\mathbb{Q})$
vanishes and $X$ admits a Ricci-flat K\"ahler orbifold metric. 
\end{Th}

Before we prove the above theorem, we describe our construction 
in more detail.  We call a singularity that looks locally as 
$\mathbb{C}^3/\Gamma$ where $\Gamma$ is generated by 

\begin{equation}
\left(\,
\begin{array}{ccc}
\zeta_3^a & 0 & 0 \\
0 & \zeta_3^b & 0 \\
0 & 0 & \zeta_3^c \\
\end{array}
\,\right)
\end{equation}

\emph{of type $\tfrac{1}{3}(a,b,c)$}. If the fixed locus of $\rho$ 
contains isolated points and curves, $Z$ has singularities of kind

\begin{equation}
\tfrac{1}{3}(2,2,2)\:, \quad
\tfrac{1}{3}(1,0,2)\:, \quad
\tfrac{1}{3}(2,2,1)\:, \quad
\tfrac{1}{3}(1,0,1)\:,
\end{equation}

since $\psi$ acts as multiplication with $\zeta_3^2$ on $T_0
\mathbb{P}^1$ and as multiplication with $\zeta_3$ on 
$T_{\infty} \mathbb{P}^1$. 

In the first two cases, we have $\Gamma\subset SL(3,\mathbb{C})$. 
Since we are in complex dimension $3$, a crepant resolution of
these singularities exists. It is known that there exists a unique 
crepant resolution of singularities of type $\tfrac{1}{3}(2,2,2)$,
see for example \cite[p. 496]{Roan}. Its exceptional divisor is 
$\mathbb{P}^2$. The singularities along the curves are 
ADE-singularities of type $A_2$ that can be blown up. The 
exceptional divisor is a bundle over the singular curves whose 
fibers are two $\mathbb{P}^1$s that intersect at a single point. 

Unfortunately, the singularities of type $\tfrac{1}{3}(2,2,1)$ and
$\tfrac{1}{3}(1,0,1)$ are not Gorenstein. Since the determinant
of the generator of the groups $\Gamma$ that correspond to these 
singularities is $\zeta_3^2$, the canonical bundle $\mathcal{K}_{X_0}$
is not a line bundle in the classical sense but has some singular fibers of
type $\mathbb{C}/\mathbb{Z}_3$. Nevertheless, $\mathcal{K}_{X_0}^3$
is a smooth line bundle since $\Gamma$ acts trivially on the
fibers of $\mathcal{K}^3_{S\times\mathbb{P}^1}$ at the fixed points. Since the 
third power of $\mathcal{K}_{X_0}$ is a line bundle, but not $\mathcal{K}_{X_0}$ itself, the
first Chern class of $X_0$ is an element of $H^2(X_0,\tfrac{1}{3}\mathbb{Z}) 
\subset H^2(X_0,\mathbb{Q})$ rather than $H^2(X_0,\mathbb{Z})$. 

The third step of our construction is a modification of the following lemma 
that can be found in Chapter I.17 of \cite{BHPV}. 

\begin{Le}
\label{CoveringLemma}
Let $M_0$ be a connected, complex variety and $D$ be a
reduced divisor of $M_0$, i.e.

\[
D = \sum_i D_i
\]

where the $D_i$ are irreducible subvarieties of codimension
one. We assume that there exists a line bundle $\mathcal{L}$ 
on $M_0$ such that

\begin{equation}
\label{BundleCondition}
\mathcal{L}^n = \mathcal{O}_{M_0}(D)\:.
\end{equation}

Then there exists an $n$-fold ramified covering $\pi:M\rightarrow M_0$
with branch-locus $D$. 
\end{Le}

\begin{proof}
There exists a global holomorphic section $s$ of $\mathcal{L}^n$ that 
has zeroes of order $1$ at the $D_i$ and is non-zero outside of the 
$D_i$. Let $\alpha: \mathcal{L}\rightarrow \mathcal{L}^n$ be the map with 
$\alpha(v) := v^{\otimes n}$. $M$ can be defined as the subspace
$\alpha^{-1}(s(M_0)) \subseteq \mathcal{L}$ and $\pi: M\rightarrow M_0$ 
is the restriction of the projection map $\mathcal{L} \rightarrow M_0$. 
\end{proof}

\begin{Rem}
The ramified cover from the above lemma is not ne\-cessarily 
unique since there exist several line bundles satisfying 
(\ref{BundleCondition}) if $H^2(M_0,\mathbb{Z})$ has $n$-torsion.
\end{Rem}

We are now able to prove Theorem \ref{MainThm}.

\begin{proof}
The main part of the proof is to show that $\mathcal{K}^3_X$ is trivial. Let 
$z\in\mathbb{P}^1\setminus \{0,\infty\}$ be arbitrary. 
$\mathcal{O}_{\mathbb{P}^1}(\{z\})$ is the hyperplane line 
bundle $H$. The canonical bundle $\mathcal{K}_{\mathbb{P}^1}$ is 
isomorphic to $H^{-2}$. By extending the transition functions of 
$H$ canonically from $\mathbb{P}^1$ to $S\times \mathbb{P}^1$
we obtain a line bundle $H'$ on $S\times \mathbb{P}^1$. Since 
$S$ has vanishing first Chern class, the canonical bundle of 
$S\times\mathbb{P}^1$ is  $H'^{-2} \cong
\mathcal{O}_{S\times \mathbb{P}^1}(S\times\{z\})^{-2}$. 

There exists a $(3,0)$-form $\eta$ on $S\times \mathbb{P}^1$ with
a pole of second order along $S\times\{z\}$. $\eta\otimes
(\rho\times\psi)^{\ast}\eta \otimes (\rho\times\psi)^{2\ast}\eta$
is a $\rho\times\psi$-invariant section of $\mathcal{K}^3_{S\times\mathbb{P}^1}$
with poles of second order along $(S\times\{z\}) \cup 
(S\times \{\zeta_3\cdot z\}) \cup (S\times \{\zeta_3^2\cdot z\})$. 
Therefore, there exists a section $\alpha$ of $\mathcal{K}^3_Z$ with $p^{\ast}\alpha
= \eta\otimes (\rho\times\psi)^{\ast}\eta \otimes 
(\rho\times\psi)^{2\ast}\eta$. $\alpha$ has a pole of second order
along $p(S\times \{z\})$ and thus we have

\begin{equation}
\mathcal{K}_Z^3 \cong \mathcal{O}_Z(p(S\times \{z\}))^{-2}\:.
\end{equation}

Since $\phi$ is a crepant resolution, we also have

\begin{equation}
\label{CanBundleX0}
\mathcal{K}_{X_0}^3 \cong \mathcal{O}_{X_0}(D_0)^{-2}\:.
\end{equation}

We define $\mathcal{L}:=\mathcal{O}_{X_0}(D_0) \otimes 
\mathcal{K}_{X_0}$. $\mathcal{O}_{X_0}(D_0)$ is a line bundle
but $\mathcal{K}_{X_0}$ has exceptional fibers at the singular locus
$R$ of $X_0$. It follows from equation (\ref{CanBundleX0}) that
$\mathcal{L}^3 \cong \mathcal{O}_{X_0}(D_0)$. We use the 
construction from the proof of Lemma \ref{CoveringLemma} to 
define a map $\pi:X \rightarrow X_0$. Since the restriction of
$\mathcal{L}$ to $X_0\setminus R$ is a line bundle, the map 
$\pi' := \pi|_{X\setminus \pi^{-1}(R)} : X\setminus \pi^{-1}(R) 
\rightarrow X_0\setminus R$ is a threefold ramified cover. 

At a point $p\in R$ the situation is more complicated. Let $\Gamma$
be the group that is generated either by $(\zeta_3^2,\zeta_3^2,
\zeta_3)$ or by $(\zeta_3,1,\zeta_3)$ and let $U$ be a 
$\Gamma$-invariant neighborhood of $0\in\mathbb{C}^3$. Near
$p$ the orbifold $X_0$ looks like $U/\Gamma$. The total space of the
line bundle $\mathcal{L}$ can be identified with $(U\times\mathbb{C})/
\Gamma'$ where $\Gamma'$ is generated either by 
$(\zeta_3^2,\zeta_3^2,\zeta_3,\zeta_3)$ or by 
$(\zeta_3,1,\zeta_3,\zeta_3)$. Since $\mathcal{L}^3$ is
an ordinary line bundle, it can be described by $(U\times\mathbb{C})/
\Gamma''$ where the generator of $\Gamma''$ is either 
$(\zeta_3^2,\zeta_3^2,\zeta_3,1)$ or by $(\zeta_3,1,\zeta_3,1)$.

The map $\alpha:\mathcal{L}\rightarrow\mathcal{L}^3$ can be
written with respect to the coordinates on $U\times\mathbb{C}$
as

\[
\alpha((z_1,z_2,z_3,z_4)\Gamma') = (z_1,z_2,z_3,z_4^3)\Gamma''
\] 

Let $s$ be a section of $\mathcal{L}^3$ that has a zero of order one 
along $D_0$ and is non-zero elsewhere. In particular, $s$ is non-zero 
on the neighborhood that we have identified with $U/\Gamma$. In this 
picture, the variety $X$ is locally given by 

\[
\{(z,\sqrt[3]{s(z)})\Gamma' | z\in U \} \cup
\{(z,\zeta_3\sqrt{s(z)})\Gamma' | z\in U \} \cup
\{(z,\zeta_3^2\sqrt{s(z)})\Gamma' | z\in U \}
\]

Geometrically this is the union of three copies of $U/\Gamma$
that intersect in 

\[
\{(z,\sqrt[3]{s(z)})\Gamma' | z\in F \}
\]

where $F$ is the fixed point set of $\Gamma$. $F$ is an isolated point
if $\Gamma$ is generated by $(\zeta_3^2,\zeta_3^2,\zeta_3)$ and a 
line in the other case.

Our considerations show that $X$ has self-intersections along a
set that is biholomorphic to $R$. The self-intersections can be 
removed by blowing up the self-intersection locus in the 
total space of $\mathcal{L}$. For reasons of convenience, we denote
the blown up space also by $X$ and the composition of the blow-up
map and $\pi:X\rightarrow X_0$ also by $\pi$. After this final step,
$X$ is an orbifold with singularities of type $\tfrac{1}{3}(2,2,1)$ or 
$\tfrac{1}{3}(1,0,1)$. 

In the situation of Lemma \ref{CoveringLemma} we have (see Chapter 
I.17 of \cite{BHPV})

\begin{equation}
\mathcal{K}_M \cong \pi^{\ast}(\mathcal{K}_{M_0} \otimes \mathcal{L}^{n-1})\:.
\end{equation}

If we were in that situation, we could conclude that

\begin{equation}
\label{TrivBundle1}
\mathcal{K}_X \cong \pi^{\ast}(\mathcal{K}_{X_0} \otimes \mathcal{L}^2) \cong
\pi^{\ast}(\mathcal{K}_{X_0}^3 \otimes \mathcal{O}_{X_0}(D_0)^2) \cong 
\mathcal{O}_X\:.
\end{equation}

Since $\mathcal{L}$ has exceptional fibers, we have to prove that 
the canonical bundle is also trivial in our situation. Let $\beta$ be a 
section of $\mathcal{K}_{X_0}^3$ with a pole of second order along 
$D_0$. By a short calculation in local complex coordinates, we can show 
that in accordance with (\ref{TrivBundle1}) the pull-back $\pi'^{\ast}\beta$ 
is a nowhere vanishing holomorphic section $\gamma$ of 
$\mathcal{K}^3_{X\setminus \pi^{-1}(R)}$. We obtain no additional 
zeroes or poles if we extend $\gamma$ meromorphically
to $\pi^{-1}(R)$. The reason for this is that $\pi^{-1}(R)$ is the disjoint
union of subvarieties of codimension $\geq 2$ and that the zero set
of a holomorphic section is a subvariety of codimension $1$.
All in all, we have proven that $\gamma$ has neither zeroes nor
poles and therefore $\mathcal{K}_X^3$ is trivial. 

Finally, we have to prove that $X$ is K\"ahler. Let $g$ be a K\"ahler metric on 
$S\times\mathbb{P}^1$. $g + (\rho\times\psi)^{\ast}g + 
(\rho\times\psi)^{2\ast} g$ is a $\rho\times\psi$-invariant K\"ahler metric on
$S\times\mathbb{P}^1$. $Z$ therefore carries an orbifold K\"ahler
metric. Among the crepant resolutions of a K\"ahler orbifold there exists at
least one that admits a K\"ahler metric \cite[p. 137]{Joyce}. Since
our resolution $\phi$ is unique, $X_0$ thus has to be K\"ahler. It is known
\cite{Abe} that a ramified cover of a complete K\"ahler space also admits a 
K\"ahler metric. Therefore, $X$ is K\"ahler, too. Since the Calabi
conjecture is also true for orbifolds \cite[p. 135]{Joyce}, $X$ admits a 
Ricci-flat K\"ahler orbifold metric.
\end{proof}

\begin{Rem}
If we had divided $S\times\mathbb{P}^1$ by an automorphism of
order $p$ and had defined $\mathcal{L}$ as $[D_0]\otimes \mathcal{K}_{X_0}$,
we would have $\mathcal{L}^p = [D_0]^{p-2}$. Since we need
$\mathcal{L}^p = [D_0]$ to construct a $p$-fold ramified cover,
the condition $p=3$ in our theorem is necessary.
\end{Rem}

\section{The Euler characteristic}

We go through each step of our construction and determine the
Euler characteristic of the intermediate threefolds. In the end, we 
obtain the Euler characteristic of our Calabi-Yau orbifolds. Since
$\chi(S)=24$ for any K3 surface $S$, we have

\begin{equation}
\chi(S\times \mathbb{P}^1) = 48 
\end{equation}

It is well-known that the Euler characteristic of a quotient of a manifold 
$M$ by a finite group $G$ is given by

\begin{equation}
\label{EulerQuotient}
\chi(M/G) = \frac{1}{|G|} \sum_{g\in G} \chi(M^g)\:,
\end{equation}

where $|G|$ is the cardinality of $G$ and $M^g$ is the fixed point set of 
the action of $g\in G$ on $M$. Since $\psi$ has exactly two fixed points
and the Euler characteristic of the fixed point set of $\rho$ is $24 - 3m$
\cite{ArSa}, we obtain 

\begin{equation}
\begin{aligned}
\chi(Z) & = \frac{1}{3}\left(\chi(S\times\mathbb{P}^1) + 
2\chi\left((S\times\mathbb{P}^1)^{\rho\times\psi} \right) \right) \\
& = \frac{1}{3} \left(48 + 4(24 - 3m) \right) \\
& = 48 - m
\end{aligned}
\end{equation}

Let $Y$ be a variety and let $f:\widetilde{Y} \rightarrow Y$
be a resolution of a singularity along a subvariety $W$ of $Y$.
The exceptional divisor we denote by $E$. We have

\begin{equation}
\chi(\widetilde{Y}) = \chi(Y) + \chi(E) - \chi(W)\:.
\end{equation}  

As we have mentioned in the previous section, the exceptional divisor
of the crepant resolution of the isolated singular points of $Z$ is 
$\mathbb{P}^2$. Its Euler characteristic is $3$. The crepant resolution of 
the singular curves replaces the singular curves by a bundle over those curves
whose fibers are two intersecting $\mathbb{P}^1$s. The Euler characteristic 
of a fiber is again $3$. Since we resolve the singularities at $n$ isolated points
and along $k$ rational curves and one curve of genus $g$, we have 

\begin{equation}
\begin{aligned}
\chi(X_0) & = \chi(Z) + 3n + 3(2k + 2 - 2g) - n - (2k + 2 - 2g)\\
& = 48 - m + 2(n + 2k + 2 - 2g) \\
& = 48 - m + 2(24 - 3m) \\
& = 96 - 7m 
\end{aligned}
\end{equation}

Finally, we are able to use (\ref{EulerQuotient}) again and obtain

\begin{equation}
\chi(X) = 3 \chi(X_0) - 2\chi(D_0) = 240 - 21m 
\end{equation}

We go through the table from Theorem \ref{ClassThm} and obtain the
following numerical values:

\begin{center}
\begin{tabular}{l|c}
$\quad (m,a)\qquad$ & $\qquad \chi(X)\qquad$ \\

\hline

\quad (1,1) & 219 \\

\quad (2,0),(2,2) & 198 \\ 

\quad (3,1),(3,3) & 177 \\ 

\quad (4,2),(4,4) & 156 \\

\quad (5,1),(5,3),(5,5) & 135 \\ 

\quad (6,0),(6,2),(6,4),(6,6) & 114 \\

\quad (7,1),(7,3),(7,5),(7,7) & 93 \\

\quad (8,2),(8,4) & 72 \\

\quad (9,1),(9,3) & 51 \\

\quad (10,0),(10,2) & 30 \\
\end{tabular}
\end{center}


\begin{thebibliography}{999}
    \bibitem{Abe} Abe, M.; Jin, T.; Shima, T.: Ramified coverings over a
    complete K\"ahler space. Arch. Math. 83, 154 - 158, 2004.

    \bibitem{ArSa} Artebani, M.; Sarti, A.: Non-symplectic automorphisms of order 3 on $K3$ 
    surfaces. Math. Ann. 342, No.4, 903-921, 2008.

    \bibitem{ArSaTa} Artebani, M.; Sarti, A.; Taki, S.: $K3$ surfaces with non-symplectic   
    automorphisms of prime order. Math. Z. 268, 507-533, 2011. 

    \bibitem{BHPV} Barth, W.; Hulek, K.; Peters, C.; van de Ven, A.: Compact complex
    surfaces. Second Enlarged Edition. Springer-Verlag, Berlin Heidelberg, 2004.

    \bibitem{Bor} Borcea, C.: Calabi-Yau threefolds and complex multiplication. In:
    Essays on mirror manifolds, 489 - 502. Int. Press, Hong Kong, 1992.

    \bibitem{CG} Cattaneo, A.; Garbagnati, A.: Calabi-Yau threefolds of 
    Borcea-Voisin type and elliptic fibrations, arXiv:1312.3481.

    \bibitem{Cynk} Cynk, S.: Cyclic coverings of Fano threefolds. Ann. Pol. Math. 80,
    117 - 124, 2003. 
    
    \bibitem{Dillies} Dillies, Jimmy: Generalized Borcea-Voisin Construction.
    Lett. Math. Phys. 100, Issue 1, 77 - 96, 2012.   

    \bibitem{Garbagnati} Garbagnati, Alice: New families of Calabi-Yau threefolds 
    without maximal unipotent monodromy. Manuscr. Math. 140, No. 3-4, 273-294, 
    2013.

    \bibitem{Joyce} Joyce, D.: Compact Manifolds with Special Holonomy. 
    Oxford University Press, New York, 2000.

    \bibitem{KoLe} Kovalev, A.; Lee, N.-H.: $K3$ surfaces with 
    non-symplectic involution and compact irreducible $G_2$-manifolds.
    Math. Proc. Camb. Phil. Soc. 151, 193-218, 2011.

    \bibitem{Roan} Roan, S.-S.: Minimal resolutions of Gorenstein orbifolds
    in dimension three. Topology 35 No. 2, 489-508, 1996.

    \bibitem{Rohde1} Rohde, Jan Christian: Cyclic coverings, Calabi-Yau manifolds 
    and complex multiplication. Lecture Notes in Mathematics 1975. Springer, Berlin,
    2009.
    
    \bibitem{Rohde2} Rohde, Jan Christian: Maximal automorphisms of Calabi-Yau 
    manifolds versus maximally unipotent monodromy. Manuscr. Math. 131, No. 3-4,
    459-474, 2010.

    \bibitem{Taki} Taki, S.: Classification of non-symplectic automorphisms of order
    3 on $K3$ surfaces. Math. Nachr. 284, No. 1, 124-135, 2011.

    \bibitem{Voi} Voisin, C.:  Miroirs et involutions sur les surfaces K3. Journ\'{e}es 
    de G\'{e}om\'{e}trie Alg\'{e}brique d'Orsay, Orsay, 1992. Ast\'{e}risque 218, 
    273 - 323, 1993.
\end{thebibliography}
\end{document}